\documentclass[draft]{amsart}

\usepackage{amssymb}
\usepackage{url}

\newtheorem{theorem}{Theorem}
\newtheorem*{lemma*}{Lemma}
\newtheorem{lemma}[theorem]{Lemma}
\newtheorem{corollary}[theorem]{Corollary}

\newtheorem{atheorem}{Theorem}

\newtheorem{alemma}{Lemma}[atheorem]

\theoremstyle{remark}
\newtheorem{rem}[theorem]{Remark}

\newtheorem{ex}[theorem]{Example}

\def\vfi{\varphi}
\def\si{\sigma}

\def\N{\Bbb N}

\def\C{\Bbb C}
\def\D{\Bbb D}
\def\Ga{\Gamma}
\def\ga{\gamma}
\def\th{\theta}

\def\om{\omega}

\def\be{\beta}
\def\ity{\infty}

\def\al{\alpha}
\def\la{\lambda}

\def\arg{\mathop{\mbox{\rm arg}}}
\def\Arg{\mathop{\mbox{Arg}}}

\def\arcsin{\mathop{\mbox{arcsin}}}

\newcommand{\rr}[1]{re^{i{#1}}}
\newcommand{\dff}{\stackrel{\rm def}=}
\newcommand{\ti}{\tilde}

\begin{document}






\title[Argument of  analytic functions and Frostman's  conditions]{Argument of bounded analytic functions and Frostman's type conditions}


%

\author{Igor  Chyzhykov}
\address{Faculty of Mechanics and Mathematics,
Lviv National University,
Uni\-ver\-sy\-tets'ka  1, 79000, Lviv, Ukraine}

\email{ichyzh@lviv.farlep.net}


%

\thanks{}

%

\begin{abstract}
We describe the growth of the naturally defined argument of
a bounded analytic function in the unit disk in terms of the
complete measure introduced by A.Grishin. As a consequence, we characterize the local behavior of a logarithm of  an analytic function. We also find necessary and sufficient conditions for closeness of $\log f(z)$, $f\in H^\infty$, and the local concentration of the zeros of $f$.
\end{abstract}

\maketitle



\section{Introduction}
One of the basic theorems in complex analysis is the Argument
principle, which states that  if $f(z)$ is a meromorphic function inside
and on some closed contour $\gamma$, with $f$ having no zeros or poles
on $\gamma$, then the increase of  $\Arg  f(z)$ along $\gamma$ divided over $2\pi$  is equal to
$N-P$, where $N$ and $P$ denote respectively the number of zeros
and poles of $f(z)$ inside the contour $\gamma$. It seems reasonable to
ask what can be said if the number of zeros (poles) of $f$ is
infinite. Obviously, the contour should contain a singular point and  the increase of  $\Arg  f(z)$ along $\gamma$
need not be bounded in this case. Theorem \ref{t:arg} of this
paper can be considered as a generalization of the Argument
principle for bounded analytic functions in the unit disk $ \D =
\left\{ {z \in \C :\left| z \right| < 1} \right\} $. We compare
the growth of the naturally defined argument of a bounded analytic
function $F$ with the distribution of its complete measure in the
sense of A.Grishin \cite{Gr1,FG}.

Let us introduce some notation.
We write $D(\zeta, \rho)=\{\xi\in \C: |\xi-\zeta|<\rho\}$ The symbols
$C(\cdot)$ and $K(\cdot)$  stand for some positive constants
depending on values in the parentheses, not necessarily the same in each
occurence.
Let $H^\infty$ be the class of bounded analytic functions in $\D$.
It is well-known \cite{Ha, CL} that $f\in H^\infty$, $|f(z)|<C$,
$z \in \D$, $C>0$, can be represented in the form
\begin{equation}\label{e:fbound}
f(z)=Cz^p\tilde{B}(z)g(z),
\end{equation}
 where $p$ is nonnegative integer, $\tilde B$ is the Blaschke
product constructed by the zeros of $f$,
\begin{equation}\label{e:bla}
  \tilde B(z) = \prod\limits_{n = 1}^\infty   \frac{\overline {a_n} {(a_n  -
z)}}{|a_n|{(1 - z\overline {a}_n ) }}\equiv \prod_{n=1}^\infty
\frac{b(z,a_n)}{|a_n|}, \quad  a_n\ne 0, \; \sum_n (1-|a_n|)<\infty,
\end{equation}
 and $g_\psi$ is an analytic function without zeros of the form
\begin{equation}\label{e:noze}
g_\psi(z)=\exp \left\{ { - \frac{1}{{2\pi }}\int\limits_{ - \pi
}^\pi {\frac{{e^{it}  + z}}{{e^{it}  - z}}d\psi^* (t)} +iC'}
\right\},
\end{equation}
where $ \psi^* $ is a non-decreasing function, and $C'$ is a real constant.

We shall also consider the product
\begin{equation}\label{e:bprod}
 B(z) = \prod_{n = 1}^\infty
b(z,a_n)
\end{equation}
 which  differs from $\tilde B(z)$ by a constant factor,
provided that the Blaschke condition (\ref{e:bla}) holds. $B(z)$ converges
almost everywhere to a finite limit $B(e^{i\th})$ as $z $ tends to
$e^{i\th}$ non-tangentially; moreover, $|\ti B(e^{i\th})|=1$.

For a fixed $\th_0$ the following theorem of O.\ Frostman
\cite{CL,Fr} gives  necessary and sufficient conditions for existence
of the radial limit of $\tilde B(z)$.

\begin{atheorem}\label{t:fr}
Necessary and sufficient that
\begin{equation}\label{e:rl}
  \lim_{r \uparrow 1}  f(\rr{\th_0})=L
\end{equation}
and $|L|=1$ for $f=\ti B$, and every subproduct of $\tilde B(z) $,
is that
\begin{equation}\label{e:fr}
 \sum\limits_{k = 1}^\infty  {\frac{{1 - \left| {a_k }
\right|}}{{\left| {e^{i\theta _0 }  - a_k } \right|}}}<\infty.
\end{equation}
\end{atheorem}

If we drop the condition $|L|=1$, then the theorem holds for $ B$ instead
of $\ti B$ as well.

Theorem \ref{t:fr} was generalized and complemented by many
authors \cite{Ca,AC,Co}. Usually one uses the condition
\begin{equation}\label{e:fr1}
 \sum\limits_{k = 1}^\infty
  \frac{1 - \left| a_k \right|}{\left| e^{i\theta _0}  - a_k
   \right|^{1-\ga}}<\infty
\end{equation}
with $\ga\le 0$ instead of (\ref{e:fr}). We note that if (\ref{e:fr1}) holds with $\ga\le 0$ and $|a_n -e^{i\th_0}|<1$, then  there is
only a finite number of zeros $a_n$ in any Stolz angle with the
vertex $e^{i\th_0}$   where the Stolz angle
with the vertex $\zeta$ is defined by $$S_\si(\zeta)=\{ \zeta\in
\D: |1-z\bar\zeta|\le \si(1-|z|)\},\quad \si\ge 1,$$ provided that (\ref{e:fr1}) is valid. We are interested
in the case when (\ref{e:fr}) fails to hold, but (\ref{e:fr1})
hold, when $\gamma \in (0,1]$.  The limit cases $\ga=1$ and $\ga=0$
correspond to the Blaschke condition and the Frostman condition,
respectively. In this situation the zeros of $B$ can be accumulated on
the radius ending at $e^{i\th_0}$, which is impossible when
$\gamma\le 0$. Thus, $\arg B(z)$ should be defined carefully.
If we want to obtain lower estimates for $|B(z)|$, $z\to
e^{i\th_0}$, $z\in \D$, we must exclude exceptional sets including the zero set.

Relations between conditions on the zeros of a Blaschke product $B$
and the membership  of $\arg B(e^{i\th})$ in $L^p$ spaces, $0<p\le
\infty$, were investigated in \cite{Ry}. Criteria for boundedness of $p$-th integral means, $1\le p<\infty$, of
$\log |B|$ and $\log B$ were established in \cite{VM}.

Since the proof of the necessity of Theorem \ref{t:fr} is based on
estimates  of the argument, one may ask whether it is possible to
describe the zero distribution of a Blaschke product in terms of
the behavior of $\arg B(z)$. A simple example shows that it is not
sufficient to know the radial behavior of the argument.

Let $(a_n)$ be an arbitrary  Blaschke sequence with non-real
elements. We define $c_{2n-1}=a_n$, $c_{2n}=\bar a_n$. Then $$
B(r)=\prod _{n=1}^\infty b(r,c_n)=\prod _{n=1}^\infty
\frac{|a_n|^2 |a_n-r|^2}{|1-ra_n|^2}, \quad 0\le r< 1. $$ Thus, $$\arg
B(r)\dff\sum_{n=1}^\infty \arg b(r,c_n)\equiv 0, \quad 0\le r< 1.$$

But a situation is quite different if we consider the behavior of
$\arg B(z)$ in a Stolz angle  $S_\si(\zeta)$, $\zeta\in \partial
\D$, $1< \si<+\infty$, $S_\si=S_\si(1)$. Then we are able to
describe the zero distribution, and even the distribution of the so-called
complete measure in the sense of A. Grishin \cite{Gr1,FG}.

Let $SH^\infty (\D)$ be the class of  subharmonic functions in
$\D$ bounded from above. In particular, $\log|f|\in SH^\infty
(\D)$ if $f\in H^\infty$. Every function $u\in SH^\infty (\D)$
which is harmonic in a neighborhood of the origin can be
represented in the form (cf. \cite[Ch.3.7]{HK})
\begin{equation}\label{e:uzob}
  u(z)=\int_{\D} \log\frac{|b(z, \zeta)|}{|\zeta|} d\mu_u(\zeta)- \frac1{2\pi}\int_{\partial \D} \frac{1-|z|^2}
  {|\zeta-z|^2}
  d\psi(\zeta),
\end{equation}
where $\mu_u$ is the Riesz measure of $u$ \cite{HK}, and $\psi$ is
a Borel  measure on the unit circle. A \emph{complete measure
$\lambda_u$ of $u$ in the sense of Grishin} is defined \cite{Gr1,FG} by the boundary measure  and the Riesz measure
of $u(z)$. But, since \cite{CL} $$\lim_{r\uparrow1}
\int_{\theta_1}^{\theta_2} \int_{\D} \log\frac{|b(\rr\th,
\zeta)|}{|\zeta|} d\mu_u(\zeta) d\th =0, \quad -\pi\le
\th_1<\th_2\le \pi,$$
i.e. the boundary values of the first integral in (\ref{e:uzob})  do not contribute to the boundary measure,
we can define $\lambda_u$ of a Borel set
$M\subset \overline{\D}$ such that $M\cap\partial{\D}$ is measurable with respect to the Lebesgue measure on $\partial \D$ by
\begin{equation}\label{e:cm}
  \lambda_u(M)= \int_{\D\cap M} (1-|\zeta|)\, d\mu_u(\zeta) + \psi(M\cap \partial \D).
\end{equation}
The measure $\lambda=\lambda_u$ has the following properties:

\begin{itemize}
  \item [(1)] $\lambda$ is finite on $\overline\D$;
  \item [(2)] $\lambda$ is non-negative;
  \item  [(3)] $\lambda$ is a zero
measure outside $\overline{\D}$; \item[(4)]
$d\lambda\Bigr|_{\partial \D}(\zeta)=d\psi(\zeta)$; \item[(5)]
$d\lambda\Bigr|_{\D}(\zeta)= \ (1-|\zeta|)\, d\mu_u(\zeta)$.
\end{itemize}

If $u=\log|f|$, $ f\in H^\infty$, then we shall write $\la_f$
instead of $\la_{\log|f|}$. If $\tilde B$ is a Blaschke product of
form (\ref{e:bla}), then $\lambda_{\ti B}(M)= \sum_{a_n\in M}
(1-|a_n|)$.

We shall say that $g$ is a {\it divisor} of  $f\in H^\ity$ if
$g\in H^\ity$ and  there exists an $h\in H^\ity$ such that $f=gh$.
It is easy to see, that in this case we have
$\lambda_g(M)+\lambda_h(M)= \lambda_f(M)$ for an arbitrary Borel
subset $M$ of $\overline\D$ such that $M\cap\partial{\D}$ is measurable.

The following generalization of Frostman's result on bounded
functions is valid.
\begin{atheorem}[Lemma 3, \cite{AC}] \label{t:ah_cl}
\sl Let $F\in H^\infty$, and $\lambda_F(\{\zeta\})=0$ for some
$\zeta\in
\partial \D$. The following  are equivalent.
\begin{itemize}
  \item [1)] $$\int_{\overline{\D}}
  \frac{d\lambda_F(\xi)}{|\zeta-\xi|}<\infty.$$
  \item [2)] Every divisor of $F$ has a radial limit at $\zeta$.
\end{itemize}
\end{atheorem}

\section{Main results and examples}

  Without loss of generality we can consider the local asymptotic behavior in a neighborhood of $\zeta=1$ ($\th_0=0$). Let
$A(z,\xi)=\frac{1-|\xi|^2}{1-z\bar \xi}$,  $\arg w$ be the
principal branch of $\Arg w$.

\begin{lemma} \label{l:arg} Let $\xi\in\D$, $z\in \D\setminus\{\xi\}$. Then
$|\arg b(z,\xi)|\le \pi \min \Bigl\{ |A(z,\xi)|,1\Bigr\}$.
\end{lemma}

Consider the product $B(z)$ defined by \eqref{e:bprod}.
We  make radial cuts   $l_n=\{\zeta\in \D: \zeta=\tau a_n, \tau\ge
1\}$. The region $\D^*=\D\setminus \bigcup _{n=1}^\infty l_n$
contains no zeros of $B(z)$.  Due to Lemma \ref{l:arg}  we define
(cf. \cite{Ry}) a continuous branch $$\log B(z)\stackrel{\rm def}=
\sum_{n=1}^\infty \log b(z, a_n), \quad z\in \D^*,$$  $\arg
B(z)\stackrel{\rm def}=\Im \log B(z)$.  In particular, we have $\log B(0)=0$ and $\arg (B_1B_2)=\arg B_1+\arg B_2$, where $B_1$, $B_2$ are Blaschke products.
Later, in the proof of Lemma \ref{l:arg}, we
also define $\arg B(z)$ on the cuts except zeros. But the resulting function
will not be continuous there.

In order to formulate our results we need some information on
fractional derivatives \cite[Chap.IX]{Dj}, \cite[Chap.8]{Z}. For
$h\in L(0,a)$ (integrable in the sense of Lebesgue on $(0,a)$) the
fractional integral of Riemann-Liouville $h_\al$ of order $\al>0$
is defined by the formula
 \cite{HL,Dj,Z}
$$h_\al(r)=D^{-\al} h(r)=\frac 1{\Ga(\al)}\int _0^r (r-x)^{\al-1}
h(x)\, dx, \quad r\in (0,a),$$
 $$D^0h(r)\equiv
h(r), \quad D^\al h(r)=\frac{d^p}{dr^p} \{ D^{-(p-\al)} h(r)\}, \;
\quad \al\in (p-1,p],\; p\in \mathbb{N},$$ where $\Gamma(\alpha)$
is the Gamma function. The function $h_\al$ is continuous for  $\al\ge 1$, and
coincides with a primary function of the correspondent order when
$\al\in\N$. We note that for $\alpha<0$ the operator $D^\al$ is
associative and commutative as a function of $\alpha$. When writing
$D^{-\al}f(z)$ we always mean  that the operator is taken on the
variable $r=|z|$.

Let $S_\si^*(\zeta)=S_\si(\zeta)\cap D(\zeta, \frac12)$. The following theorem yields a necessary and sufficient condition for the  local growth of $\arg  f$ in terms of the generalized Frostman's condition  for the complete measure in the sense of Grishin of  a bounded analytic function in the unit disk.

\begin{theorem}\label{t:arg}
Let $F$ be a bounded analytic function in $\D$,  $0\le \gamma<1$,
$\zeta_0\in \partial \D$. In order that for every  divisor $f$ of
$F$ and every $\sigma>1$ there exist a constant $K=K(\ga, \si,
F)>0$ such that
\begin{equation}\label{e:int}
  \sup _{z\in S_\si^*(\zeta_0)} |D^{-\ga} \arg f(z)|<K,
\end{equation}
it is necessary and sufficient that
\begin{equation}\label{e:fr2}
\int_{\overline \D}
  \frac{d\lambda_F(\zeta)}{\left|\zeta_0 - \zeta
   \right|^{1-\ga}}<\infty.
\end{equation}
\end{theorem}

\begin{rem} Since (\ref{e:int}) must hold for every
divisor $f$ of $F$, (\ref{e:int}) is equivalent  to
\begin{equation}\label{e:int'}
  \sup _{z\in S_\si^*(\zeta_0)} D^{-\ga} |\arg f(z)|<K
\end{equation}
for every divisor $f$ and every $\si>1$.  In fact, we shall prove
that $(\ref{e:int})\Rightarrow (\ref{e:fr2}) \Rightarrow
(\ref{e:int'})$. Since it is evident that (\ref{e:int'}) implies
(\ref{e:int}), this will prove Theorem \ref{t:arg}.
\end{rem}

\begin{rem} As we shall see, in order that (\ref{e:fr2}) hold it is sufficient that
(\ref{e:int}) holds for a finite number of divisors  of a special
form. Moreover, it is enough to require that $$\varlimsup_{z\to
\zeta_0, z\in \Gamma_j} |D^{-\ga} \arg f(z)|<+\infty,$$ for
two particular segments $\Gamma_j$ ending at $\zeta_0$, $\Gamma_j\subset
\D\cup \{\zeta_0\}$, $j\in \{1, 2\}$.
\end{rem}

\begin{corollary}\label{t:b_arg}
Let $B$ be a Blaschke product defined by (\ref{e:bprod}), $0\le
\gamma<1$, $\zeta_0\in \partial \D$. In order that for every
subproduct $B^*$ of $B$ and every $\sigma>1$ there exist a
constant $K=K(\ga, \si, B)>0$ such that
\begin{equation}\label{e:int_b}
  \sup _{z\in S^*_\si(\zeta_0)} |D^{-\ga} \arg B^*(z)|<K,
\end{equation}
it is necessary and sufficient that
\begin{equation}\label{e:fr_bl}
 \sum\limits_{k = 1}^\infty
  \frac{1 - \left| a_k \right|}{\left|\zeta_0 - a_k
   \right|^{1-\ga}}<\infty.
\end{equation}
\end{corollary}

\begin{corollary} \label{c:cc} Let $F\in H^\infty$, $0\le \ga <1$. If (\ref{e:fr2}) holds,
then for every divisor $f$ of $F$ the function  $\arg f(r)$ is bounded if $\ga=0$, and belongs to the convergence class of order $\ga$ if $\ga\in (0,1)$, i.e. $$  \int_0^1
(1-r)^{\ga-1} |\arg f(r)| dr<+\infty.$$
\end{corollary}

\begin{proof}[Proof of Corollary \ref{c:cc}]
In fact, if $0<\ga<1$, then
\begin{gather*} \sup_{0< r<1} D^{-\ga} |\arg f(r)|=
\sup_{0< r<1} \frac1{\Ga(\ga)} \int_0^r (r-x)^{\ga-1} |\arg f(x)|
dx\ge
\\ \ge \sup_{0< r<1} \frac1{\Ga(\ga)}\int_0^r (1-x)^{\ga-1} |\arg f(x)| dx= \frac1{\Ga(\ga)} \int_0^1
(1-x)^{\ga-1} |\arg f(x)| dx.
\end{gather*}
The  case $\ga=0$ follows from Theorem \ref{t:fr}.
\end{proof}

Since for any $\si>1$ we have $\D\subset \bigcup_{|\zeta|=1} S^*_\si(\zeta) \cup \overline{D}(0,\frac12)$, from Theorem \ref{t:arg} we get the following corollary.
\begin{corollary}\label{c:unif_arg}
Let $F$ be a bounded analytic function in $\D$,  $0\le \gamma<1$, and
$\zeta_0\in \partial \D$. Then for
\begin{equation*}
  \sup _{z\in \D} |D^{-\ga} \arg f(z)|<\infty
\end{equation*}
to hold, it  is necessary and sufficient that
\begin{equation*}
\sup_{\zeta_0\in \partial\D}\int_{\overline \D}
  \frac{d\lambda_F(\zeta)}{\left|\zeta_0 - \zeta
   \right|^{1-\ga}}<\infty.
\end{equation*}
\end{corollary}


\begin{ex} The analytic function $$F(z)=\exp\Bigl\{-\frac {1+z}{1-z} \Bigr\}, \quad z\in\D,$$
shows that the condition $\la_F(\{\zeta\})=0$ in Theorem \ref{t:ah_cl} is essential. In fact, we have $\la_F(\zeta)=\delta(\zeta-1)$ where $\delta(\zeta-1)$  is the unit mass supported at $\zeta=1$. The function $F$ has the non-tangential limit  $0$ as $z\to 1$, $z\in \D$, but 
\begin{equation}\label{e:la_ga}
\int_{\overline{\D}}\frac {d\la_F(\xi)}{|\xi-1|^{1-\ga}}=\infty, \quad \ga <1.
\end{equation}
We have
\begin{equation*}
    \arg F(\rr\vfi)=\Im \Bigl\{- \frac{1+\rr\vfi}{1-\rr\vfi}\Bigr\}=-\frac {2r\sin \vfi}{|1-\rr\vfi|^2}.
\end{equation*}
It is clear that $\arg F(\rr{\beta(r-1)})\to +\infty $ as $r\uparrow1$ for any   positive constant $\be$.
Theorem \ref{t:arg} yields that $D^{-\ga} \arg F(z)$ is unbounded for any $\ga<1$, consequently
$$\arg F(z)\ne O\Bigl( \frac1{(1-|z|)^\ga}\Bigr), \quad z\to1, z\in S_\si, \si>1, \ga<1.$$
The last relation follows from the fact that $h(r)=O((1-r)^{-\ga})$ $(r\uparrow1)$ implies  $D^{-\ga_1}h(r)=O(1)$ $(r\uparrow1)$ provided $\ga<\ga_1<1$ (cf. Lemma \ref{l:gaint} and the lemma from \cite{Ch1}).
\end{ex}

\begin{ex} Let $\al\in [0,1)$, $$
\psi^*(t)=\left\{
          \begin{array}{ll}
            t^{1-\al}, & t\in [0,\pi] \\
            -|t|^{1-\al}, & t\in [-\pi,0].
          \end{array}
        \right.$$
Consider the function  $g(z)=g_\psi(z)$ defined by (\ref{e:noze}), where $C'=0$. Then $g$ is analytic, bounded and has no zeros in $\D$. In this case $\la_{g}\Bigr|_{\D}$ is the zero measure, and $d\la_{g}(e^{it})=d\psi(t)$, $t\in [-\pi,\pi]$. We have
\begin{gather} \label{e:arg_ex_2}
    \int_{\overline\D} \frac{d\la_{g}(\zeta)}{|\zeta-1|^{1-\ga}}=\int_{-\pi}^\pi \frac {d\psi^*(t)}{|e^{it}-1|^{1-\ga}}=2(1-\al)\int_0^\pi \frac{dt}{t^\al|e^{it}-1|^{1-\ga}}.
\end{gather}
Since $|e^{it}-1|\sim t$ as $t\downarrow0$ the integral from (\ref{e:arg_ex_2}) is convergent if and only if the integral
$\int_0^\pi t^{-1-\al+\ga} dt $ is convergent.

Thus,  if $\ga >\al$ we have $$ D^{-\ga} \arg g_\psi(z)=O(1), \quad z\to 1, z\in S_\si, \si>1.$$
In the limit case $\ga=\al=0$ one can show that $$\arg g(r)\asymp \log \frac 1{1-r}, \quad  r\uparrow1.$$
\end{ex}

Now we consider the local behavior of the logarithm of a bounded function.
Following C.N.Linden \cite{Li_m} we introduce characteristics of
concentration of zeros. Let $n_z(h)$ be the number of zeros of
an analytic function $f$ in $\overline{D}(z,h(1-|z|))$, $$ N_z(h)=\sum_{|a_n-z|\le
h(1-|z|)} \ln \frac{h(1-|z|)}{|z-a_n|}=\int\limits_0^{(1-|z|)h}
\frac{n_z(s)}s ds.$$
These quantities are usually used for characterizing the local behavior of the modulus  of an analytic function \cite{Hay_min}, \cite{Kras}.

\begin{theorem}\label{t:lnb}
Let $F\in H^\infty$, $0\le \gamma<1$, $0<h<1$, and $\zeta_0\in
\partial \D$. In order that for every divisor $f$ of $F$ and every
$\sigma>1$ there exist a constant $K=K(\ga, \si, F)>0$ such that
\begin{equation}\label{e:intr}
  \sup _{z\in S^*_\si(\zeta_0)} |D^{-\ga} (\log f(z) +N_z(h))|<K,
\end{equation}
it is necessary and sufficient that
\begin{equation}\label{e:fr2r}
\int_{\overline{\D}}
\frac{d\lambda_F(\zeta)}{|\zeta_0-\zeta|^{1-\ga}}<\infty.
\end{equation}
\end{theorem}

\begin{corollary}\label{c:lnb}
Let $B$ be a Blaschke product defined by (\ref{e:bla}), $0\le
\gamma<1$, $\zeta_0\in
\partial \D$, $0<h<1$. In order that for every subproduct $B^*$ of
$B$ and every $\sigma>1$ there exist a constant $K=K(\ga, \si, B)>0$
such that
\begin{equation*}\label{e:intrB}
  \sup _{z\in S_*(\zeta_0)} |D^{-\ga} (\log B(z) +N_z(h))|<K
\end{equation*}
it is necessary and sufficient that
\begin{equation*}
 \sum\limits_{k = 1}^\infty
  \frac{1 - \left| a_k \right|}{\left|\zeta_0 - a_k
   \right|^{1-\ga}}<\infty.
\end{equation*}
\end{corollary}

Statements of such type can be used for obtaining  estimates for
the minimum modulus of analytic and subharmonic functions
(\cite{Hay_min,Kras,Li_m}), but  we omit this topic here.

If $F$ has no zeros, we easily obtain

\begin{corollary}\label{c:no_zer}
Let $g\in H^\infty$ be of the form (\ref{e:noze}), $0\le
\gamma<1$, $\zeta_0\in
\partial \D$. In order that for every divisor $g^*$ of $g$ and
every $\sigma>1$ there exist a constant $K=K(\ga, \si, g)>0$ such
that
\begin{equation*}
  \sup _{z\in S^*_\si(\zeta_0)} |D^{-\ga} \log g^*(z) |<K,
\end{equation*}
it is necessary and sufficient that
\begin{equation}\label{e:fr3r}
\int_{\partial{\D}}
\frac{d\psi(\zeta)}{|\zeta_0-\zeta|^{1-\ga}}<\infty,
\end{equation}
where $\psi$ is the Stieltjes measure generated by $\psi^*$.
\end{corollary}

Let $\psi$ and $\chi$ be  Borel  measures on $\partial \D$. We
shall write that $\chi\prec \psi$ if $\chi (M) \le \psi(M)$ for an
arbitrary Borel set $M\subset\partial \D$. Note that $g_\chi$ is a
divisor of $g_\psi$ if and only if $\chi\prec \psi$.

Applying Corollary \ref{c:no_zer} and Theorem \ref{t:arg} to the function $g_\psi(z)=\exp\{ h_\psi(z)\}$ of form (\ref{e:noze}), we obtain
\begin{theorem}\label{t:eqv}
Let \begin{equation}\label{e:noze1} h_\psi(z)= \int\limits_{ - \pi
}^\pi {\frac{{e^{it}  + z}}{{e^{it}  - z}}d\psi^* (t)} ,
\end{equation}
where $ \psi^* $ is a monotone function on $[-\pi,\pi]$. Let
$0\le \ga <1$, and $\zeta_0\in
\partial \D$. Let $\psi$ be the Stieltjes measure generated by
$\psi^*$. The following conditions are equivalent:
\begin{itemize}
  \item [1)] For every Borel measure $\chi$ on $\partial \mathbb{D}$ such that $\chi \prec \psi$ and
every $\sigma>1$ there exists a constant $K=K(\ga, \si, \psi)>0$
such that
\begin{equation*}
  \sup _{z\in S^*_\si(\zeta_0)} |D^{-\ga} h_\chi(z) |<K.
\end{equation*}
  \item [2)] For every Borel measure $\chi$ on $\partial D$ such that $\chi \prec \psi$ and
every $\sigma>1$ there exists a constant $K=K(\ga, \si, \psi)>0$
such that
\begin{equation*}
  \sup _{z\in S^*_\si(\zeta_0)} |D^{-\ga} \Im h_\chi(z) |<K.
\end{equation*}
  \item [3)] Condition (\ref{e:fr3r}) holds.
\end{itemize}
\end{theorem}

\section{Proof of Theorem \ref{t:arg}} \label{s:prof_arg}

We may assume that $\zeta_0=1$. We restrict ourself to the case
$0<\ga <1$. Let $f$ be a divisor of $F$, and $f$ of form
(\ref{e:fbound}). First, we consider $\arg B(z)$, and start with proof of Lemma 1.

\begin{proof}[Proof of Lemma 1]
We consider the triangle with the vertices $A=z\bar \xi$,
$B=|\xi|^2$, $C=1$; $AB=1-|\xi|^2$, $BC=||\xi|^2-z\bar \xi|$,
$AC=|1-\bar \xi z|$. The quantity  $$\varphi_\xi=\arg b(z,\xi)=\arg
\frac {|\xi|^2-z \bar \xi}{1-z\bar \xi}$$ is the value of the  angle between
the vectors $\vec{AB}$ and $\vec{AC}$. The cut $ \{\zeta=\tau \xi:
1\le \tau \le \frac 1{|\zeta|}\}$ corresponds to $BC$. Thus,
  $|\vfi_\xi|<\pi$ if $z\bar \xi\not \in BC$. For $z\bar \xi\in
  BC$, i.e. for $z$ laying on the cut, we define by the semicontinuity
  $\vfi_\xi\stackrel{\rm
  def}=-\pi$. Therefore, $\arg b(z, \xi)$ is defined in
  $\D\setminus \{\xi\}$ but, obviously,  not continuous on the cut.

Let $D_\xi$ be the disk constructed on AB as on the diameter. We
consider two cases.

If $C=z\bar \xi \in D_\xi$, then $\pi/2< |\varphi_\xi|\le \pi$ and
$|z\bar \xi -1|\le 1-|\xi|^2$, i.e. $|A(z, \xi)|\ge 1$. Therefore
$|\arg b(z, \xi)| \le \pi=\min\{\pi, |A(z, \xi)|\}$ as required.

 If $z\bar \xi \not\in D_\xi$, then $|\varphi_\xi|\le \pi/2$.
 Thus, $\varphi_\xi=\arcsin \frac{\Im b(z,\xi)}{|b(z, \xi)|}$.
 Since $$\Im b(z, \xi)=-\Im A(z, \xi)=\Im(\bar z \xi)\frac
 {1-|\xi|^2}{|1-\bar \xi z|^2}, $$
 we have
\begin{gather} \label{e:barg}
  |\varphi_\xi|=
 \biggl|\arcsin \frac{\Im(\bar z \xi)} {||\xi|^2 -z\bar \xi|}
 \frac{1-|\xi|^2}{|1-z\bar\xi|} \biggr|\le\\ \le  \arcsin \min\Bigl\{1,
 \frac{1-|\xi|^2}{|1-z\bar\xi|}\Bigr\} \le \frac \pi 2\min\{1,
  |A(z, \xi)|\}.\nonumber
\end{gather}
\end{proof}

\begin{lemma}  \label{l:gaint} Let $0\le \ga < \al <\infty$. Then there exists a constant $C(\ga, \al)>0$ such that
\begin{equation}\label{e:gaint}
D^{-\ga} \frac{1}{|1-r\zeta|^\al} \le \frac{C(\ga, \al
)}{|1-r\zeta|^{\al-\ga}}, \quad \zeta \in \overline{\D}, 0<r<1
\end{equation}
\end{lemma}
\begin{proof}[Proof of Lemma \ref{l:gaint}]
Let $\arg \zeta =\th$. Then
\begin{equation}\label{e:le}
  |1-x\zeta|\ge |1-r\zeta|\cos (\th/2), \quad 0\le x\le r<1.
\end{equation}
In fact, geometric arguments yield  that if $|r\zeta|\le\cos \th$,
then $|1-x\zeta|\ge |1-r\zeta|$. Otherwise, $\cos\th <
|r\zeta|<1$, and we deduce $$ |1-x\zeta|\ge |1-e^{i\th} \cos\th|=
|1-e^{i\th}| \cos(\th/2)\ge |1-r\zeta| \cos(\th/2)$$ as required.

Without loss of generality we may assume that $$|\th|\le \pi/4,\quad
\frac 12<r < 1, \quad  2|1-r\zeta|<r.$$ Using (\ref{e:le}), we obtain
\begin{gather*}
D^{-\ga} \frac1{|1-r\zeta|^\al}=\frac1{\Gamma(\ga)} \int_0^r
\frac{(r-x)^{\ga-1} }{|1-x\zeta|^\al} dx= \\ =
\frac1{\Gamma(\ga)}\biggl( \int_0^{r-2|1-r\zeta|}+\int
_{r-2|1-r\zeta|}^r \biggr)\frac{(r-x)^{\ga-1} }{|1-x\zeta|^\al}
dx\le \\ \le  \frac1{\Gamma(\ga)}\biggl( \int_0^{r-2|1-r\zeta|}
\frac{(r-x)^{\ga-1} }{(1-x|\zeta|)^\al} dx+\int _{r-2|1-r\zeta|}^r
\frac{(r-x)^{\ga-1} }{|1-r\zeta|^\al\cos^\al \frac\th 2}
dx\biggr)\le \\ \le \frac1{\Gamma(\ga)}\biggl(
\int_0^{r-2|1-r\zeta|} \frac{dx}{(1-x|\zeta|)^{1-\ga+\al}}-
\frac{(r-x)^{\ga}}{\ga|1-r\zeta|^\al\cos^\al\frac\th2}
\Bigr|^r_{r-2|1-\zeta r|}\biggr)=\\
=\frac1{\Gamma(\ga)}\biggl(\frac{1}{(\al-\ga)|\zeta|}
\frac1{(1-x|\zeta|)^{\al-\ga}}\Bigr|_0^{r-2|1-r\zeta|} +
\frac{2^\ga}{\ga \cos^\al\th/2|1-r\zeta|^{\al-\ga}} \biggr)\le \\
\le \frac1{\Gamma(\ga)}\biggl( \frac{2}{\al-\ga} \frac1{(1-r
+2|1-r\zeta||\zeta|)^{\al-\ga}}+ \frac{2^{\ga+\al/2 +
1}}{\ga|1-r\zeta|^{\al -\ga}}\biggr) \le\frac{C(\ga, \al
)}{|1-r\zeta|^{\al-\ga}}.
\end{gather*}
The lemma is proved.
\end{proof}
In order to finish the proof of the sufficiency we need the
following lemma (\cite[Lemma 1]{GPV}).
\begin{alemma} \label{l:GPV} Given $\si\ge 1$ there exists a constant $C(\si)>0$ such
that $$ |1-\zeta|\le C(\si)|1-\bar z\zeta| , \quad \zeta \in \D,
\; z\in S_\si.$$
\end{alemma}

By Lemma \ref{l:arg} we have
\begin{equation*}
  |\arg B(z)|\le \pi  \sum_{n=1}^\infty
\frac{1-|a_n|^2}{|1- z\bar a_n|^{2}}\le \frac{C(F)}{1-r}.
\end{equation*}
 Using Lemmas \ref{l:arg},
\ref{l:gaint}, \ref{l:GPV} and (\ref{e:fr2}) we obtain for $z\in
S_\si$
\begin{gather}\nonumber
D^{-\ga} |\arg B(z)|\le\sum_{n=1}^\infty D^{-\ga} |\arg b(z,
a_n)|\le \\ \le \pi\sum_{n=1}^\infty D^{-\ga} \frac{1-|a_n|^2}{|1-z\bar
a_n|} \le \nonumber
 \pi C(\ga)\sum_{n=1}^\infty \frac{1-|a_n|^2}{|1-z\bar
a_n|^{1-\ga}}\le \\ \le  \pi C(\ga,\si) \sum_{n=1}^\infty
\frac{1-|a_n|^2}{|1- a_n|^{1-\ga}}< C(\ga, \si, F)<+\infty.
\end{gather}

We now consider $\arg g(z)$. In view of (\ref{e:noze}) we have
$(z=\rr\vfi)$
\begin{equation}\label{e:img}
\arg g(z)=\Im  \biggl\{ - \frac 1{2\pi} \int\limits_{ - \pi }^\pi
{\frac{{e^{it} + z}}{{e^{it}  - z}}d\psi^* (t)} \biggr\} ={ -
\frac 1{2\pi} \int\limits_{ - \pi }^\pi
{\frac{r\sin(\vfi-t)}{|e^{it} - z|^2}d\psi^* (t)} },
\end{equation}
Using Lemmas \ref{l:gaint} and \ref{l:GPV} for $z\in S_\si$, we
deduce
\begin{gather*}
D^{-\ga} |\arg g(z)|=\biggl| \frac 1{\Gamma(\ga)} \int_0^r
(r-x)^{\ga-1} dx \int_{-\pi}^\pi \frac{x \sin(\vfi-t)}{|e^{it}  -
xe^{i\vfi}|^2}d\psi^* (t)\biggr| \le \\ \le \int_{-\pi}^\pi \frac
{|\sin(\vfi-t)|}{\Gamma(\ga)} d\psi^* (t) \int_0^r
\frac{(r-x)^{\ga-1}}{|e^{it} - xe^{i\vfi}|^2} dx \le \\ \le C(\ga)
\int_{-\pi}^\pi \frac{|\sin(\vfi-t)|}{|e^{it} -
re^{i\vfi}|^{2-\ga}}d\psi^* (t) \le C(\ga) \int_{-\pi}^\pi
\frac{1}{|e^{it} - re^{i\vfi}|^{1-\ga}}d\psi^* (t)\le  \\ \le
C(\ga,\si) \int_{-\pi}^\pi \frac{1}{|e^{it} - 1|^{1-\ga}}d\psi^*
(t). \label{e:dga_up}
\end{gather*}

 The sufficiency is proved.

{\it Necessity.} First, we consider the subproduct $B^*$ of $B$
constructed by the zeros $a_n$ satisfying $\Im a_n \ge 0$,
$|1-a_n|\le \frac13 $. We denote such $a_n$ by $a_n^*$.

Let $z=\rr\vfi$ satisfy $\arg(1-z)=\pi/4$, $\zeta\in [0,z]$,
$\zeta=\rho$. In particular, $\Im \zeta<0$. Then
$$ \Im ({a}_n^*
\bar \zeta )=-\Im \zeta \Re a^*_n + \Im {a}^*_n \Re \zeta \ge 0, $$
and consequently (see (\ref{e:barg})) $$\arg b(\zeta,a_n^*)\ge \arcsin \frac{\Im(\bar
\zeta a_n^*)(1-|a_n^*|^2)}{||a_n^*|^2-\zeta \bar a_n^*||1-\bar
a_n^* \zeta|}\ge 0 .$$
  By our assumption
\begin{gather}\nonumber
  C\ge D^{-\ga} \arg B^* (z)= \sum_n D^{-\ga} \arg b(z,a_n^*)\ge
  \\ \ge
  \sum_n \int_0^r (r-t)^{\ga-1} \arcsin
\frac{\Im(te^{-i\vfi}a_n^*)(1-|a_n^*|^2)}{||a_n^*|^2-t
e^{i\vfi}\bar a_n^*||1-\bar a_n^* te^{-i\vfi}|} dt.\label{e:larg}
\end{gather}
For every  $a_n^*$ satisfying $1-|a_n^*|\ge 2(1-r)$ and $\zeta\in
[0,z]$ such that $$|1-a^*_n|\le r-\rho\le 2|1- a^*_n|$$ we
have $$ \rho\ge r-2|1-a_n^*|\ge r-\frac23 > \frac14, \quad
r\uparrow1.$$ Thus, $$|\Im \zeta|\ge |\Im z|/4\ge (1-r)/4.$$ Hence,
\begin{gather}\label{e:as2}
\Im (a_n^* \bar \zeta)\ge -\Im \zeta \Re a_n^* \ge     \frac{\Re
a^*_n}{4} (1-r). \end{gather}\ Similarly,
\begin{gather}\label{e:ans2}
\Im (a_n^* \bar \zeta)\ge \Re \zeta \Im a_n^* \ge
    \frac{\Re z}2
|1-a_n^*|,  \quad  {a_n^*\not\in S_2}.
  \end{gather}

Further,
\begin{gather}
 |a^*_n-\zeta|\le |1-a^*_n|+{2}|1-|\zeta|| =
|1-{a_n^*}| +{2} (r-|\zeta|+1-r)\le \nonumber \\ \le |1-a^*_n|
+{2}(2+\frac 12)|1-a^*_n|= 6|1-a^*_n|, \label{e:den1}
\\ |1-{\bar a}^*_n \zeta|\le  |1-{\bar a}^*_n| +|{\bar
a}^*_n-{\bar a}^*_n\zeta|\le 6|1-a^*_n|. \label{e:den2}
\end{gather}
Thus, for $a_n^*\in S_{2}$ using (\ref{e:as2}), (\ref{e:den1}), and
(\ref{e:den2}) we have
\begin{gather}
\int_0^r (r-t)^{\ga-1} \arcsin \nonumber \frac{\Im({\bar a}^*_n
te^{i\varphi} )(1-|a^*_n|^2)}{||a^*_n|^2-te^{i\varphi} {\bar
a}^*_n||1-{\bar a}^*_n te^{i\varphi}|} dt \ge
\\ \ge \nonumber \int_0^r \frac{(r-t)^{\ga-1} \Re a^*_n
(1-t)(1-|a^*_n|^2)}{144 |1-a^*_n|^2|a^*_n|} dt \ge
\\  \ge \nonumber C(\ga)
\int\limits_{r-2|1-a^*_n|}^{r-|1-a^*_n|}\frac{(r-t)^{\ga } \Re
a^*_n }{ |a^*_n| (1-|a^*_n|)} dt \ge C(\ga)\frac{1}{1-|a^*_n|}
\int\limits_{r-2|1-a^*_n|}^{r-|1-a^*_n|} (r-t)^{\ga}\, dt\ge
\\ \ge C(\ga)(1-|a^*_n|)^\ga.
\end{gather}

If $a_n^*\in \D\setminus S_2$, then using
(\ref{e:ans2})--(\ref{e:den2}) we obtain
\begin{gather}
\int_0^r (r-t)^{\ga-1} \arcsin \nonumber \frac{\Im({\bar a}_n^*
te^{i\varphi} )(1-|a^*_n|^2)}{||a^*_n|^2-te^{i\varphi} {\bar
a}_n^*||1-{\bar a}_n^* te^{i\varphi}|} dt \ge
\\ \ge \nonumber \int_0^r \frac{(r-t)^{\ga-1} \Re z
|1-a^*_n|(1-|a^*_n|^2)}{72 |1-a^*_n|^2|a^*_n|} dt \ge
\\  \ge \nonumber C(\ga)
\int\limits_{r-2|1-a^*_n|}^{r-|1-a^*_n|}\frac{(r-t)^{\ga-1 } \Re
z|1-|a^*_n||}{ |a^*_n| |1-a^*_n|} dt \ge \\ \ge
C\frac{1-|a^*_n|}{|1-a^*_n|}
\int\limits_{r-2|1-a^*_n|}^{r-|1-a^*_n|} (r-t)^{\ga-1}\, dt\ge
 C\frac{1-|a^*_n|}{|1-a^*_n|^{1-\ga}}.
\end{gather}

 Hence, $$C>D^{-\ga} \arg {B^*}(z)=\sum_n D^{-\ga} \arg
b(z,a^*_n)> C \sum_{|a_n^*|\le 1-2(1-r)}
\frac{1-|a^*_n|}{|1-a_n^*|^{1-\ga}}.$$ Since the constants $C$ are
independent of $r$, tending $r\uparrow 1$ we get the statement of the
necessity for $\arg B^*$, and consequently for $\arg B$.

Now, we have to estimate $D^\ga(\arg g_{\psi})$ from below. Let
$\psi_1$ be the restricted function of $\psi^*$ on $[0,\pi/2]$.
Let $\arg (1-z)=\frac \pi 4 $. Then
\begin{gather*} D^{-\ga} \Im g_{\psi_1}(z)=  \frac1{\Gamma(\ga)}\int_{-\pi}^\pi
\int_0^r \frac{(r-\rho)^{\ga-1}\sin (t-\vfi)}{|\rho
e^{i\vfi}-e^{it}|^2} d\rho d\psi_1(t) =\\ =\frac1{\Gamma(\ga)}
 \int_{0}^{\pi/2} \sin (t-\vfi) d\psi^*(t) \int_{0}^{r} \frac{(r-\rho)^{\ga-1}}
 {|\rho e^{i\vfi} -e^{it}|^2} d\rho.\end{gather*}
In order to estimate the inner integral we may assume that
$r>2|z-e^{it}|$ without loss of generality. For $|z-e^{it}|\le
r-\rho\le 2|z-e^{it}|$ we have $$ |\rho e^{i\vfi} -e^{it}|\le |z-
\rho e^{i\vfi}|+ |z-e^{it}|\le $$ $$ \le (1+o(1))|r-\rho|+ |z-e^{it}|\le
4|z-e^{it}|, \quad r\uparrow1.$$ Moreover, since $\arg z\sim r-1$,
 we have $t-\vfi\ge (1+o(1))(1-r)$ as $ r\uparrow1$.
Then,
\begin{gather*}
|z-e^{it}|= |r-e^{i(t-\vfi)}|\le 1-r
+1-\cos(\vfi-t)+\sin(\vfi-t)\le  \\ \le (1+o(1))\sin(1-r) +2\sin^2
\frac{t-\vfi}2 +\sin(t-\vfi)\le \\ \le(4+o(1))\sin(t-\vfi), \quad r
\uparrow1.
\end{gather*}
Using the latter estimates we deduce
\begin{gather*}
C\ge D^{-\ga} \Im g(z) \ge \frac1{\Gamma(\ga)}\int_0^{\pi/2}  \sin
(t-\vfi) d\psi^*(t) \int_{r-2|z-e^{it}|}^{r-|z-e^{it}|}
\frac{(r-\rho)^{\ga-1}}{|\rho e^{i\vfi} -e^{it}|^2} d\rho  \ge \\
\ge \int_{0}^{\pi/2} \frac{\sin (t-\vfi)}{16 |z-e^{it}|^2}
d\psi^*(t) \int_{r-2|z-e^{it}|}^{r-|z-e^{it}|} {(r-\rho)^{\ga-1}}
d\rho \ge
\\ \ge  C(\ga) \int_{0}^{\pi/2} \frac{\sin (t-\vfi)|z-e^{it}|^\ga}{
|z-e^{it}|^2} d\psi^*(t) \ge  C(\ga) \int_0^{\pi/2}
\frac{d\psi^*(t)}{|z-e^{it}|^{1-\ga}}.
\end{gather*}
Tending $r$ to 1 and using   Fatou's lemma we  conclude that $$C
\ge C(\ga) \int_0^{\pi/2} \frac{d\psi^*(t)}{|1-e^{it}|^{1-\ga}}.$$
Similarly, it can be shown that $ \int_{-\pi/2}^0
\frac{d\psi(t)}{|1-e^{it}|^{1-\ga}}<C$, and consequently,
 $$ \int_{-\pi}^\pi
\frac{d\psi(t)}{|1-e^{it}|^{1-\ga}}<C.$$
 Theorem \ref{t:arg} is proved.

\section{Proof of Theorem \ref{t:lnb} and final remarks}

\begin{proof}[Proof of Theorem \ref{t:lnb}]

 The necessity of the theorem follows
from Theorem \ref{t:arg}.

{\it Sufficiency.} Let $f$ be a divisor of $F$. Without loss of generality
we may assume that $f=Bg$, where $B$ and $g$ are defined as above. Let $L(z,h, f)=\log f(z)+N_z(h)$.
 We have
\begin{gather}\nonumber
\Re L(z, h, f)=\Re L(z,h,B)-  \frac 1{2\pi}\int_{-\pi}^\pi \Re\frac{e^{it} +z}{e^{it} -z} d\psi^*(t)=\\ =\sum_{|a_n-z|\le
h(1-r)} \ln\Bigl|\frac{a_nh(1-r)}{1-z\bar a_n} \Bigr|+ \nonumber
\sum_{|a_n-z|> h(1-r)} \ln\Bigl|\frac{a_n(z-a_n)}{1-z\bar a_n}
\Bigr|- \\ - \frac 1{2\pi}\int_{-\pi}^\pi \frac{1-r^2}{|e^{it}
-z|^2} d\psi^*(t)\le 0. \label{e:rel}
\end{gather}

Let us estimate $\Re L(z,h,f)$ from the below. For $|a_n-z|\le
h(1-r)$ we have
\begin{equation*}\label{e:ineq}
|1-z\bar a_n|=|1-|z|^2 +z(\overline{z-a_n})|\le 1-r^2+rh(1-r)\le
(2+h)(1-r),
\end{equation*}
and
\begin{equation} |a_n|\ge r-h(1-r)\ge \frac{1-h}2, \quad r\ge \frac 12. \label{e:an_low}
\end{equation}  Hence,
\begin{gather}\nonumber
\sum_{|a_n-z|\le h(1-r)} \ln \Bigl| \frac{a_nh(1-r)}{1-z\bar a_n}
\Bigr| \ge \sum_{|a_n-z|\le h(1-r)} \ln \frac{|a_n|h}{2+h} \ge \\ \ge  - C(h)
\sum_{|a_n-z|\le h(1-r)} |A(z,a_n)| , \quad r\ge \frac12,
\label{e:in7'}\end{gather}
because $C_1\le |A(z,a_n)|\le C_2$  if $|a_n-z|\le h(1-r)$.

On the other hand, (see \cite[p.13]{Ts})
\begin{equation}\label{e:in8}
\sum_{|A(z,a_n)|< \frac12} -\ln|b(z,a_n)|\le 2 \sum_{|A(z,a_n)|<
\frac12} |A(z,a_n)|.
\end{equation}

It is known that a pseudohyperbolic disk $\mathcal{D}(z,s)=\Bigl\{
\zeta: \Bigl|\frac {z-\zeta}{1-z\bar \zeta}\Bigr|<s\Bigr\}$ is the
disk $D(z^*, \rho_z(s))$, where $$z^*=\frac{(1-s^2)z}{1-s^2|z|^2},\quad
\rho_z(s)=\frac{(1-|z|^2)s}{1-s^2|z|^2}.$$ We are going to prove
that
\begin{equation}\label{e:hyper}
\mathcal{D}\Bigl(z,\frac{h}{2+h}\Bigr)\subset D(z, (1-|z|)h).
\end{equation}
It is sufficient to show that $|z^*-z|+\rho_z(s)\le h(1-|z|)$ for
$s\le h/(2+h)$. We have $(|z|=r)$
\begin{gather*}
|z^*-z|+\rho_z(s)=\frac{(1-r^2)(rs^2+s)}{1-s^2r^2}\le
\frac{2(1-r)s}{1-s}.
\end{gather*}
Thus, we arrive to the inequality $2s\le h(1-s)$, which is
equivalent to $-1\le s\le \frac h{2+h}$. Inclusion (\ref{e:hyper})
is proved. Therefore,
 for $a_n\not \in D(z, h(1-r))$
we have $$ -\ln|b(z,a_n)|\le \ln   \frac{2+h}{h|\bar a_n|}.$$
Hence, using (\ref{e:an_low})
\begin{gather}\nonumber
\sum_{\begin{substack}{|A(z,a_n)|\ge \frac12 \\ |a_n-z|>h(1-r)} \end{substack}} -\ln|b(z,a_n)|\le 2
\sum_{\begin{substack}{|A(z,a_n)|\ge \frac12 \\ |a_n-z|>h(1-r)} \end{substack}} \ln \frac{2+h}{|\bar a_n|h} \le \\ \le 4\ln
\frac{4+2h}{h(1-h)} \sum_{\begin{substack}{|A(z,a_n)|\ge \frac12 \\ |a_n-z|>h(1-r)} \end{substack}} |A(z,a_n)|.
\label{e:in8'}
\end{gather}
 It follows from (\ref{e:rel})--(\ref{e:in8'}) that $$ \Re L(z,h,
B) \ge - C(h)\sum_n |A(z,a_n)|.$$ Hence, as in the proof of the sufficiency of Theorem \ref{t:arg} (see (\ref{e:dga_up})) we
deduce for $z\in S_\si$
\begin{equation}\label{e:reb}
D^{-\ga} \Re L(z,h,B) \ge -C(h,\ga) \sum_n
\frac{1-|a_n|^2}{|1-\bar a_n z|^{1-\ga}}\ge -C(h,\ga,\si,B). \quad
\end{equation}
Further,
 $$\frac 1{2\pi}\int_{-\pi}^\pi \frac{1-r^2}{|e^{it}
-z|^2} d\psi^*(t)\le  \frac 1{2\pi}\int_{-\pi}^\pi
\frac{d\psi^*(t)}{|e^{it} -z|}. $$ Applying Lemma 2 for $z $
laying in  the Stolz angle $S_\si$, we obtain
\begin{gather}
\nonumber
  D^{-\ga} \biggl| \int_{-\pi}^\pi  \frac{1-r^2}{|e^{it}
-z|^2} d\psi^*(t)\biggr| \le   \int_{-\pi}^\pi D^{-\ga}
\Bigl(\frac{1}{|e^{it} -z|}\Bigr) d\psi^*(t) \le \\ \le C(\ga)
\int_{-\pi}^\pi \frac{d\psi^*(t)}{|e^{it} -z|^{1-\ga}} \le C(\ga,
\si) \int_{-\pi}^\pi \frac{d\psi(t)}{|e^{it} -1|^{1-\ga}}<\infty.
\end{gather}
Together with (\ref{e:reb}) this yields $D^{-\ga} \Re L(z,h, f)\ge
- C$. And, in view of (\ref{e:rel}) we, finally, have $|D^{-\ga}
\Re L(z,h, F)|\le C$.

It remains to apply Theorem \ref{t:arg}. Theorem \ref{t:lnb} is
proved.
\end{proof}

\begin{rem} Frostman type condition (\ref{e:fr2}) can be rewritten in terms of the
modulus of continuity of the complete measure. Let
$\lambda_F(\zeta, \tau)\dff\la_F(\overline{D(\zeta,\tau)})$. Then
(\ref{e:fr2}) is equivalent to
\begin{equation*}
\int_0^2 \frac{d\la_F(\zeta_0, \tau)}{\tau^{1-\ga}}<+\infty
\; \mbox{ \rm or } 
\int_0^2 \frac{d\om(\tau; \zeta_0, \la_F)}{\tau^{1-\ga}}<+\infty
\end{equation*}
where $\om(\tau; \zeta_0, \la_F)$ is the modulus of continuity of
the measure $\la_F$ at the point~$\zeta_0$.
\end{rem}

From this point of view it is interesting to compare Theorem
\ref{t:eqv} with results from \cite{Ch1}, where necessary and
sufficient conditions for growth of the maximum modulus and the
maximum of the real part  of $h_\psi$ is established in terms of
the modulus of continuity of the function $\psi^*$. Similar
results for $L^p$-metrics are obtained in \cite{Chmm}.

I would like to thank  Prof.\ G.\ Gundersen and the  participants of  the Lviv seminar on the
theory of analytic functions for valuable comments which
contribute to the improvement of the initial version of the paper.


%

\end{document}